\documentclass[preprint,12pt]{elsarticle}

\usepackage{paralist,amssymb,amsmath,tikz, wasysym}
\newtheorem{proposition}{Proposition}[section]
\newtheorem{theorem}{Theorem}[section]
\newtheorem{example}{Example}[section]

\newtheorem{remark}{Remark}[section]

\newtheorem{corollary}{Corollary}[section]
\newtheorem{lemma}{Lemma}[section]
\newenvironment{proof}{\textbf{Proof.}}{\qquad $\Box$ \bigskip }

\newenvironment{mainproof}{\textbf{Proof of Theorem \ref{thm:mr+2}.}}{\qquad $\Box$ \bigskip }

\newcommand{\npmatrix}[1]{\left( \begin{matrix} #1 \end{matrix} \right)}
\newcommand{\R}{\mathbb{R}}

\newcommand{\mr}{\mathrm{mr}}
\newcommand{\rk}{\mathrm{rk}}

\newcommand{\Mm}{\mathrm{Mm}}

\begin{document}

\begin{frontmatter}



\title{The maximum of the minimal multiplicity of eigenvalues of symmetric matrices whose pattern is constrained by a graph}


\author[po]{Polona Oblak\corref{cor1}}
\ead{polona.oblak@fri.uni-lj.si}
\address[po]{Faculty of Computer and Information Science, University of Ljubljana, Ve\v cna pot 113, 
SI-1000 Ljubljana, Slovenia}
\cortext[cor1]{Corresponding author}

\author[he]{Helena \v Smigoc}
\ead{helena.smigoc@ucd.ie}
\address[he]{School of Mathematical Sciences, University College Dublin, Belfield, Dublin 4, Ireland}

\fntext[fn1]{This work was supported by Science Foundation Ireland under Grant 11/RFP.1/MTH/3157.}

\begin{abstract}
In this paper we introduce a parameter $\Mm(G)$, defined as the maximum over the minimal multiplicities of eigenvalues among all symmetric matrices corresponding to a graph $G$. We compute $\Mm(G)$  for several families of graphs.
\end{abstract}


\end{frontmatter}



\section{Introduction}

Given a simple undirected graph $G=(V(G),E(G))$ with vertex set $V(G)=\{1,2,\ldots,n\}$, let $S(G)$ be the set of all real symmetric $n \times n$ matrices $A=(a_{ij})$ such that, for $i \neq j$, $a_{ij} \neq 0$ if and only if $(i,j) \in E(G)$. There is no restriction on the diagonal entries of $A$.   

The graph $G$ pozes some conditions on the eigenvalues of $A \in S(G)$, and several questions in the literature are trying to better understand those conditions. The most general is the question of characterizing all lists of real numbers $$\{\lambda_1,\lambda_2,\ldots,\lambda_n\}$$ that can be the spectrum of a matrix $A \in S(G),$ and is known as the \emph{Inverse eigenvalue problem for $G$}. This question and the related question of characterizing all possible multiplicities of eigenvalues of matrices in $S(G)$ have been studied primarily for trees \cite{MR2996934,MR1902112,MR3005270,MR3034496}.
A subproblem to the inverse eigenvalue problem for graphs that has attracted a lot of attention over the years is that of minimizing the rank of all $A \in S(G)$. Finding the minimal rank of $G$, defined as
 $$\mr (G)=\min \{\rk(A); A \in S(G)\},$$ 
is equivalent to finding the maximal multiplicity of an eigenvalue of $A \in S(G)$, denoted by $M(G).$
\emph{The minimum rank problem} has been resolved for several families of graphs. 
We refer the reader to an excellent survey paper on the problem \cite{MR2350678} where additional references can be found.  A more recent survey paper \cite{2011arXiv1102.5142F} not only gives an up-to-date on the minimum rank problem, but it also talks about several  of its variants that can be found in the literature. 
For example,  the possible inertia of matrices $A \in S(G)$  has been studied in \cite{MR2547901,MR2781688,MR2596445} and the minimum number of distinct eigenvalues in \cite{MR3118943}. For a matrix $A$, we let $q(A)$ denote the number of distinct eigenvalues of $A$. For a graph $G$, we define
$$q(G) = \min\{q(A); \; A \in S(G)\}.$$

In \cite{MR3208410} we considered the problem of
determining for which graphs $G$ there exists a matrix in $S(G)$ whose characteristic polynomial is a square, i.e. the multiplicities of
all its eigenvalues are even. This question is closely related to the question of determining for which graphs $G$ there exists a matrix in $S(G)$ with all the multiplicities of eigenvalues at least $2$.   
In this paper we bring this topic further by defining and studying a new parameter for a graph $G$ denoted by $\Mm(G)$. For a matrix $A\in M_n(\R)$ we denote $\Mm(A)$ to be the minimal eigenvalue multiplicity of $A$. Then $\Mm(G)$ is defined to be:
$$\Mm(G)=\max\{\Mm(A); A\in S(G)\}.$$ 
In order words, we define $\Mm(G)$ to be the maximum over the minimal 
multiplicities of eigenvalues among all $A\in S(G)$. Clearly, $\Mm(G) \leq \lfloor \frac{n}{2} \rfloor$ for all nonempty 
graphs on $n$ vertices. If $\Mm(G) =1$, then 
all matrices in $S(G)$ must have a simple eigenvalue.  Moreover, it is clear that  we can determine an upper bound for $\Mm(G)$ in terms of $M(G)$ and $q(G)$ as follows:
\begin{equation}\label{eqn:basic}
\Mm(G) \leq M(G), \; \Mm(G) \leq \left\lfloor \frac{|G|}{q(G)}\right\rfloor.
\end{equation}
As we will see later in the paper both bounds can be achieved for some but not all graphs $G$. It is clear from (\ref{eqn:basic}) that graphs that have large $\Mm$, will have to have relatively small $q$, and that the two parameters are related. 

To determine $\Mm(G)$ for a given graph $G$ we need to control the multiplicities of all the eigenvalues of $A \in S(G)$, and in doing so we further our understanding of the Inverse eigenvalue problem for $G$. 

Our paper is organised as follows. In Section \ref{Sec 2} we introduce the notation and some preliminary result that give introductory insight into $\Mm(G)$ and mostly follow from known results. In Section \ref{Sec Const} we recall a few constructions from the literature and develop some generalizations that we need later.  In Section \ref{Sec 3} we look for graphs with large $\Mm$. We show that the equality $\Mm(G) = \lfloor \frac{n}{2} \rfloor$ is achieved for practically
all graphs with $\mr_+(G)=2.$ We also determine $\Mm(G)$ for complete bipartite graphs. In Section \ref{Sec 4} we look at the graphs with small $\Mm(G)$. We choose from the graphs $G$ on $n$ vertices with $\mr(G)=n-2$, and we show that among those graphs there exist graphs with $\Mm(G)=1$ and graphs with $\Mm(G)=2$.  

\bigskip

%%%%%%%%%%%% 
 \section{Notation and preliminary results} \label{Sec 2}
 
 By $M_n(\R)$ we denote the set of all $n \times n$ matrices with real entries. By $I_n$ we denote the identity matrix in $M_n(\R)$ and by $0_n$  we denote the zero matrix in $M_n(\R)$.

Our notation concerning graphs is as follows. For a graph $G=(V(G),E(G))$ we denote its order by $|G|=|V(G)|$.  The complement  $G^c$ of a graph $G$ is the graph on vertices $V(G)$ such that two vertices are adjacent in $G^c$ if and only if they are not adjacent in $G$.  The join $G \vee H$ of $G$ and $H$ is the graph union $G \cup H$ together with all the possible edges joining the vertices in $G$ to the vertices in $H$.
 The complete graph on $n$ vertices  will be denoted by $K_n$ and complete bipartite graph on disjoint sets of cardinality $m$ and $n$ by $K_{m,n}$. 
 
 The Cartesian product $G \square H$ of graphs $G, H$ is a graph with the vertex set $V(G) \times V(H)$ and $((u_1,u_2),(v_1,v_2)) \in E(G \square H)$ if and only if either $u_1=v_1$ and  $(u_2,v_2) \in E(H)$ or $u_2=v_2$ and  $(u_1,v_1) \in E(G)$. 
 The tensor product of two graphs is  defined as
$G \times H=(V(G \times H), E(G \times H))$. where $V(G \times H)=V(G) \times V(H)$
and $((u,u'),(v,v')) \in E(G \times H)$ if and only if $ (u,v) \in E(G)$ and $(u',v') \in E(H)$.
The strong product $G \boxtimes H$ of graphs $G, H$ is a graph with the vertex set $V(G) \times V(H)$ and $((u_1,u_2),(v_1,v_2)) \in E(G \boxtimes H)$ if and only if either $u_1=v_1$ and  $(u_2,v_2) \in E(H)$ or $u_2=v_2$ and  $(u_1,v_1) \in E(G)$ or $(u_1,v_1) \in E(G)$ and $(u_2,v_2) \in E(H)$.

We start with some basic observations that follow from the known results.
%First we relate $\Mm(G)$ to parameters $M(G)$ and $q(G).$ 

\begin{proposition}\label{n/2}
Let $G$ be a graph with  $|G|=2n$, $\mr(G)=n$ and $q(G)=2$. Then $\Mm(G)=n$.
\end{proposition}

\begin{proof}
Since $q(G)=2$, there exists a symmetric matrix $A \in S(G)$ with exactly two distinct eigenvalues. Let $k$ and $2n-k$ be the corresponding multiplicities. Since maximal multiplicity of an eigenvalue among matrices in 
$S(G)$ is $n$, we have $k \leq n$ and $2n-k\leq n$, proving that $k=n$.
\end{proof}

Note that the same argument as in the proof above, tells us that there is no graph $G$, $|G|=2n+1$, with  $\mr(G)=n+1$ and $q(G)=2$.

 In the following proposition we gather observations on $\Mm(G)$ that follow from results in  \cite{MR3118943} on $q(G).$
 
\begin{proposition}
\begin{enumerate}
\item For a connected graph $G$ we have $\Mm(G \vee G)=|G|$.
\item If $|G|=2n$, $|H|=2m$ with $\Mm(G)=n$ and $\Mm(H)=m$, then $\Mm(G \cup H)=n+m$.
\item Let $G'$ be the corona graph of $G$, i.e. the graph with $2 |G|$ vertices obtained from $G$ by joining each vertex of $G$ with a pendant vertex. Then
  $\Mm(G') \geq \Mm(G).$
\item Let $G$ be a graph on $n$ vertices with $q(G)=2$. Then $\Mm(G \square K_2)=n$.  
%\item \he{Lemma 6.3 can also be translated to $\Mm$.}
\end{enumerate}
\end{proposition}

\begin{proof}
\begin{enumerate}
\item  In  \cite[Theorem 5.2]{MR3118943} it is proved that $q(G \vee G)=2$ for a connected graph $G$. This is proved by showing that there exists a matrix  $Q \in S(G \vee G)$ of the form
   $$Q=\left[\begin{array}{cc}\sqrt{P} & \sqrt{I-P} \\ \sqrt{I-P} & -\sqrt{P}\end{array}\right].$$
Since $Q^2=I$ this implies that $q(Q)=2$. More precisely, $1$ and $-1$ are the only eigenvalues of $Q$. Since the trace of $Q$ is 0, it follows that both $1$ and $-1$ must have multiplicity $|G|$ and thus $\Mm(G \vee G)=|G|$.
\item Choose $A\in S(G)$ with eigenvalues $\lambda_{1}, \lambda_{2}$, each having multiplicity $n$, and $B \in S(H)$ with eigenvalues $\mu_{1}$, $\mu_{2}$, each having multiplicity $m$.
Then the matrix $$A\oplus \frac{\lambda_{1}-\lambda_{2}}{\mu_{1}-\mu_{2}} \left(B+\frac{\mu_{1}\lambda_{2}-\mu_{2}\lambda_{1}}{\lambda_{1}-\lambda_{2}} I\right)$$
has eigenvalues $\lambda_{1}$ and $\lambda_{2}$, each having multiplicity $n+m$.
\item %Follows from the proof of in \cite[Lemma 6.10]{MR3118943} in the following way.
%Let $m_1 \geq m_2 \geq \ldots \geq m_k$ be the multiplicities of eigenvalues of matrix $A \in S(G)$. 
It was proved in  \cite[Lemma 6.10]{MR3118943} that if $\lambda$ is an eigenvalue of $A$ with multiplicity $m$, then the matrix
$$B=\left[\begin{array}{cc}A & I\\ I & 0 \end{array}\right] \in S(G')$$
has two distinct eigenvalues $\mu_1(\lambda)$ and $\mu_2(\lambda)$ both with multiplicity $m$.
\item By \cite[Lemma 2.3]{MR3118943} there exists $A \in S(G)$ such that $q(A)=q(G)=2$ and the eigenvalues of $A$ are $1$ and $-1$. Therefore the minimal polynomial of $A$ is equal to $x^2-1.$ As in the proof of   \cite[Theorem 6.7]{MR3118943}, we construct a matrix
 $B=\npmatrix{A & I_n \\ I_n & -A} \in S(G \square K_2).$
 Since $A^2=I_n$, it follows that $B^2=2I_{2n}$. 
 This proves that the minimal polynomial of matrix $B$ is $x^2-2$ and thus the only possible eigenvalues of $B$ are $\sqrt{2}$ and $-\sqrt{2}$. Since the trace of $B$ is equal to $0$, the multiplicities of the two eigenvalues have to be equal and so  $\Mm(G \square K_2)=n$.
\end{enumerate}
\end{proof}

\begin{proposition}
Let $Q_s$ be the hypercube, $|Q_s|=2^s$. Then $\Mm(Q_s)=2^{s-1}$.
\end{proposition}

\begin{proof}
By \cite[Corollary 6.9]{MR3118943} we have $q(Q_s)=2$ and by \cite[Theorem 3.1]{MR2388646} we have $\mr(Q_s)=2^{s-1}$. % and $\mr(K_t \square P_2)=t$. 
Using Theorem  \ref{n/2} it follows that
$\Mm(Q_s)=2^{s-1}% \text{ and } \Mm(K_t \square P_2)=t.
$.
\end{proof}

\begin{theorem}\label{thm:tree}
Let $G$ be a graph on $n$ vertices with an induced tree on $n-k$ vertices. Then $\Mm(G)\leq k+1$.
\end{theorem}

\begin{proof}
Suppose $A \in S(G)$ with $\Mm(A)=\Mm(G)$ and let $A_1 \in S(T)$ be  its principal $(n-k) \times (n-k)$ 
submatrix for some induced tree $T$ of $G$ on $n-k$ vertices. If $\Mm(G)\geq k+2$, then all the eigenvalues of $A$ have multiplicity at least $k+2$. So, let 
us denote the eigenvalues of $A$ by 
$$\lambda_1=\ldots=\lambda_{k+2}\geq \lambda_{k+3}\geq \ldots\geq \lambda_{n}.$$
By interlacing, the submatrix $A_1$ has eigenvalues
$$\lambda_1=\lambda_{2}\geq \mu_{3}\geq \ldots\geq \mu_{n-k},$$
which contradicts \cite[Remark 4]{MR1902112} that states that  the largest eigenvalue of a tree must be simple.
\end{proof}

Corollary \ref{cor:small} will show that even for $k=1$ there exist graphs for which the inequality in Theorem \ref{thm:tree} is strict. Moreover, Corollary \ref{parallel} will show that the converse of the above theorem is not true. 

In the next proposition we gather some estimates for $\Mm$ for different products of graphs. The proofs of those estimates are not difficult and are straightforward modifications of the proofs of Theorem 4.1, Theorem 4.2 and Theorem 4.3 in \cite{MR3208410}. 

\begin{proposition}
For any two graphs $G$ and $H$ we have:
\begin{enumerate}
\item $\Mm(G \times H) \geq \Mm(G)\Mm(H)$
\item $\Mm(G \square H) \geq \Mm(G)\Mm(H)$
\item  $\Mm(G \boxtimes H) \geq \Mm(G)\Mm(H)$
\end{enumerate}
\end{proposition}

\section{Some Constructions}\label{Sec Const}

In this section we leave our discussion of the parameter $\Mm$ to develop some technical tools that we will need later in the paper to construct matrices with large $\Mm.$ First we state a construction introduced by Fiedler \cite{MR0364288}.

\begin{theorem}[\cite{MR0364288}]\label{thm:Fiedler}
Let $A$ be a symmetric $m \times m$ matrix with eigenvalues $\alpha_1,\ldots,\alpha_m$ and let $u$, $||u||=1$, be a unit eigenvector corresponding to $\alpha_1$. Let $B$ be a symmetric $n \times n$ matrix with eigenvalues $\beta_1, \ldots, \beta_n$ and let $v$, $||v||=1$, be a unit eigenvector corresponding to $\beta_1.$

Then for any $\rho$, the matrix
 $$C=\npmatrix{A & \rho uv^T \\ \rho vu^T & B}$$
 has eigenvalues $\alpha_2,\ldots,\alpha_m, \beta_2, \ldots,\beta_m, \gamma_1,\gamma_2$, where $\gamma_1, \gamma_2$ are eigenvalues of the matrix 
 $$\hat{C}=\npmatrix{\alpha_1 & \rho \\ \rho & \beta_1}.$$  
\end{theorem}

Next we offer a generalization of Theorem \ref{thm:Fiedler}.  

\begin{theorem}\label{thm:genFiedler}
Let $A \in M_n(\R)$ and $B\in M_m(\R)$ be  symmetric matrices, and  let $U=\npmatrix{U_1 & U_2} \in M_n(\R)$ and $V=\npmatrix{V_1 & V_2} \in M_m(\R)$ be orthogonal matrices such that
$$U^T AU=\npmatrix{D_1 & 0 \\ 0 & D_2} \; \text{ and }\; V^T BV=\npmatrix{E_1 & 0 \\ 0 & E_2},$$
where $D_1 \in M_{k}(\R)$ and $E_1 \in  M_l(\R)$ and the partition of the matrices $U$ and $V$ indicated above is consistent with the orders of $D_i$ and $E_i,$ $i=1,2$. 
Let 
 $$C=\npmatrix{A &U_1 R V_1^T \\  V_1 R^T U_1^T & B} \, \text{ and } \, W=\npmatrix{U_1 &0 & U_2 &0 \\ 0&V_1 & 0 &V_2}$$
for any $R \in \R^{k \times l}$.  Then $W$ is an orthogonal matrix and 
 $$W^TCW=\npmatrix{D_1 & R & 0 & 0 \\ 
  R^T & E_1 & 0 & 0\\
  0&0&D_2&0\\
  0&0&0&E_2}.$$
\end{theorem}

\begin{proof}
 Since $U$ and $V$ are orthogonal matrices, it follows that 
 $$\npmatrix{I_k&0\\0&I_{n-k}}=\npmatrix{U_1^T\\U_2^T} \npmatrix{U_1 & U_2}=\npmatrix{U_1^T U_1 & U_1^TU_2\\ U_2^T U_1 & U_2^T U_2}$$
 and hence 
 \begin{equation}\label{eq:U}
 U_1^T U_1=I_k, \, U_2^T U_2=I_{n-k}, \, U_1^TU_2=0, \, U_2^T U_1=0.
 \end{equation}
  Similarly we get:
  \begin{equation}\label{eq:V}
 V_1^T V_1=I_l, \, V_2^T V_2=I_{m-l}, \, V_1^TV_2=0, \, V_2^T V_1=0.
 \end{equation}
 Moreover, from  
 $$U^T AU=\npmatrix{D_1 & 0 \\ 0 & D_2} \text{ and } V^T AV=\npmatrix{E_1 & 0 \\ 0 & E_2}$$ we get:
 \begin{equation}\label{eq:UAU} 
 U_1^T A U_1=D_1,\, U_2^T A U_2=D_2, \, U_1^TAU_2=0, \,U_2^T A U_1=0, \, 
 \end{equation}
 and 
 \begin{equation}\label{eq:VBV}
   V_1^T B V_1=E_1, \, V_2^T B V_2=E_2, \, V_1^T BV_2=0, \, V_2^TB V_1=0. 
 \end{equation}
 Using equations (\ref{eq:U}), (\ref{eq:V}), (\ref{eq:UAU}) and (\ref{eq:VBV}) we can easily show that $W^TW=I_{n+m}$ and 
 $$W^TCW=\npmatrix{D_1 & R & 0 & 0 \\ 
  R^T & E_1 & 0 & 0\\
  0&0&D_2&0\\
  0&0&0&E_2},$$ which completes the proof.
\end{proof}

In the discussion below, we are assuming the definitions and notations from Theorem \ref{thm:genFiedler}. In our application, we will take:
\begin{equation}\label{eqn:FiedlerChoice}
D_0=\npmatrix{D_1 & R \\ R^T & E_1}=\npmatrix{a_1 & 0 & b \cos \alpha & b \sin\alpha \\
0 & a_1 & -b \sin \alpha & b \cos \alpha \\ b \cos \alpha & -b \sin \alpha & a_2 & 0 \\
b \sin \alpha & b \cos \alpha & 0 & a_2},
\end{equation}
and the following lemma deals with this situation. 

\begin{lemma}\label{D0}
 If $b=\sqrt{t(a_1-a_2+t)}$, where $a_1 > a_2-t$ and  $t>0$, then
the eigenvalues of matrix $D_0$, defined in \eqref{eqn:FiedlerChoice}, are equal to $a_1+t,a_1+t, a_2-t,a_2-t$ and the orthogonal matrix that diagonalizes $D_0$ is equal to
 \begin{equation}\label{eqn:U0}
U_0=\frac{1}{\sqrt{a_1-a_2+2t}}\npmatrix{ \frac{b \sin\alpha}{\sqrt{t}}&\frac{b \cos\alpha}{\sqrt{t}}& -\sqrt{t}\sin\alpha&-\sqrt{t}\cos\alpha \\
                                                    \frac{b \cos\alpha}{\sqrt{t}}  &-\frac{b \sin\alpha}{\sqrt{t}}&-\sqrt{t}\cos\alpha&\sqrt{t}\sin\alpha \\
                                                   0   & \sqrt{t}& 0 &\frac{b}{ \sqrt{t}}\\
                                                     \sqrt{t}  & 0 & \frac{b}{ \sqrt{t}} & 0  }.
\end{equation} 
\end{lemma}

\begin{proof}
The statements in this lemma can be checked by a straightforward direct calculation. 
%Note that for matrix $R=\npmatrix{b \cos \alpha & b \sin\alpha \\  -b \sin \alpha & b \cos \alpha}$ we have $RR^T=b^2 I$ and by the
%use of Schur decomposition we compute the determinant $\det(D_0-\lambda I)=\left(\lambda-a_1-t\right)^2\left(\lambda-a_2+t\right)^2$.
%%$$D_0-\lambda I=\npmatrix{D_1-\lambda I & R \\ R^T & E_1-\lambda I}=
%% \npmatrix{(a_1-\lambda) I & R \\ R^T & (a_2-\lambda) I}$$ as
%% \begin{align*}
%%  \det(D_0-\lambda I) &= (a_2-\lambda)^2 \cdot \det\left((a_1-\lambda) I - R\left( (a_2-\lambda) I\right)^{-1}R^T\right)=\\
%%    &=(a_2-\lambda)^2\cdot \det\left((a_1-\lambda) I - \left(a_2-\lambda\right)^{-1}RR^T\right)=\\
%%    &=(a_2-\lambda)^2\cdot \det\left((a_1-\lambda) I - \frac{b^2}{a_2-\lambda}I \right)=\\
%%    &=(a_2-\lambda)^2\cdot \left(a_1-\lambda- \frac{b^2}{a_2-\lambda}\right)^2=\\
%%    &=\left(\lambda^2-\lambda(a_1+a_2)-t(a_1-a_2+t)\right)^2=\\
%%    &=\left(\lambda-a_1-t\right)^2\left(\lambda-a_2+t\right)^2
%% \end{align*}
%% and so it follows that the eigenvalues of $D_0$ are equal to $a_1+t,a_1+t, a_2-t,a_2-t$.
% 
%It is a straightforward calculation that matrix $U_0$, defined in \eqref{eqn:U0} is an orthogonal matrix and that
%$$U_0^TD_0U_0=\npmatrix{a_1+t & 0 & 0 & 0 \\0 & a_1+t & 0 & 0 \\0 & 0 & a_2-t & 0 \\0 & 0 & 0 & a_2-t \\}$$
%and so the lemma follows.
 \end{proof}

%Now, the following result follows directly from Theorem \ref{thm:genFiedler} and Lemma \ref{D0}.
%
%\begin{lemma}\label{l2}
Let $A \in M_n(\R)$ and $B\in M_m(\R)$ be  symmetric matrices, such that $$U^T AU=\npmatrix{
a_1 I_2 & 0 \\ 0   & D_2} \; \text{, }\; V^T BV=\npmatrix{a_2 I_2&  0 \\  0 & E_2}$$
for some orthogonal matrices $U=\npmatrix{U_1 & U_2} \in M_n(\R)$ and $V=\npmatrix{V_1 & V_2} \in M_m(\R)$ as in Theorem  \ref{thm:genFiedler}.
%Furthermore, we assume that matrices $U_1$ and $V_1$ do not have any zero elements. 
Let us define $$R=\npmatrix{b \cos \alpha & b \sin\alpha \\  -b \sin \alpha & b \cos \alpha}$$ with $b=\sqrt{t(a_1-a_2+t)}$, $a_1 > a_2-t$ and 
$t>0$.
Now Theorem  \ref{thm:genFiedler} tells us that the matrix
 $$C=\npmatrix{A &U_1 R V_1^T \\  V_1 R^T U_1^T & B}$$
is orthogonally similar to 
 $$\npmatrix{(a_1+t)  I_2&  0 & 0 & 0 \\ 
 0 & (a_2-t)I_2 & 0 & 0\\
 0&0&D_2&0\\
 0&0&0&E_2}$$
with orthogonal similarity  $\hat W=\npmatrix{U_1 &0 & U_2 &0 \\ 0&V_1 & 0 &V_2}\npmatrix{U_0 & 0 \\ 0 & I}$, where 
$U_0$ is defined in \eqref{eqn:U0}.
%\end{lemma}

%\begin{proof}
%By Theorem \ref{thm:genFiedler} and Lemma \ref{D0} the eigenvalues of $D_0$ 
%are equal to  $a_1+t,a_1+t, a_2-t,a_2-t$  and an orthogonal matrix that diagonalizes $D_0$ is equal to $U_0$ as in \eqref{eqn:U0}.
%
% Let $M=\npmatrix{U_1 & 0 \\ 0 & V_1},$ then orthogonal \he{check that they are orthogonal} eigenvectors of $C$ corresponding to $a_1+t$ are the first two columns of $M$, and the ones corresponding to $a_2-t$ are the last two columns of $M$.  
%\end{proof}

%\begin{remark}
If we assume that $U_1$ and $V_1$ have no zero entries, then the first four columns of $\hat W$ also do not have any zero entries, for all but a finite set of $\alpha$. Since we may choose $\alpha$ arbitrarily, 
we have found a matrix $C$ with two orthogonal eigenvectors corresponding to $a_1+t$ that do not have any zero entries, 
and two orthogonal eigenvectors corresponding to $a_2-t$ that do not have any zero entries.
%\end{remark}

This discussion yields the following corollary. 
%The following is a direct corollary of Lemma \ref{l2}.

\begin{corollary}\label{cor:one_step}
Let $B_i \in M_{n_i}(\R),$ $i=1,2,$ be a symmetric matrix with the spectrum $(\lambda_i, \lambda_i,\sigma_i),$ where $\sigma_i$ is a list of $n_i-2$ real numbers. Let $t>0$ and assume that $B_i$  has at least two orthogonal 
eigenvectors corresponding to $\lambda_i$ that do not contain any zero elements and that $\lambda_1> \lambda_2-t$. 

Then there exists an $n_1 \times n_2$ matrix $S$ that does not contain any zero elements, such that  
 $$C=\npmatrix{B_1 & S \\ S^T & B_2} \in M_{n_1+n_2}(\R)$$
 has the spectrum $(\lambda_1+t, \lambda_1+t, \lambda_2-t, \lambda_2-t, \sigma_1, \sigma_2)$. Moreover, $C$ has at least two orthogonal eigenvectors without zero entries corresponding to $\lambda_1+t$ and at lest two orthogonal eigenvectors without zero entries corresponding to $\lambda_2-t$.
\end{corollary}

We will also make use of the following construction from \cite{MR2098598}.

\begin{lemma}[\cite{MR2098598}]\label{HS04}
Let $B$ be a symmetric $m \times m$ matrix with eigenvalues $\mu_1, \mu_2,\ldots,\mu_m$, and let $u$ be an eigenvector corresponding to $\mu_1$ normalized so that $u^Tu=1$. Let $A$ be an $n \times n$ symmetric matrix with a diagonal element $\mu_1$
 \begin{equation}\label{c1f1}
  A=\npmatrix{A_1 & b \\
               b^T & \mu_1}
 \end{equation}
 and eigenvalues $\lambda_1, \ldots, \lambda_n.$
 Then the matrix 
  \[C=\npmatrix{A_1 & bu^T \\
               ub^T & B}\]
               has eigenvalues $\lambda_1,\ldots,\lambda_n,\mu_2,\ldots, \mu_m$. 
\end{lemma}

\begin{remark}\label{rem:eigenvectors}
In \cite{MR2098598} the eigenvectors of $C$ in terms of the eigenvectors of $A$ and $B$ are given in the following way. Let 
 $$\npmatrix{v_i \\ \alpha_i}, \, v_i \in \R^{n-1}, \, \alpha_i \in \R$$
 be the orthonormal set of eigenvectors of $A$ corresponding to $\lambda_i,$ $i=1,2,
 \ldots,n$, and let $u_i,$ $i=2,\ldots,m,$ together with $u$ be the orthonormal set of eigenvectors corresponding to $\mu_i$. Then $C$ has the following eigenvectors: eigenvector of $C$ corresponding to  eigenvalue $\lambda_i,$ $i=1,2,\ldots,n$, is equal to:
  $$\npmatrix{v_i \\ \alpha_i u},$$
  and eigenvector of $C$ corresponding to $\mu_i,$ $i=2,\ldots,m$, is equal to:
  $$\npmatrix{0 \\ u_i}.$$    
\end{remark}

\section{Graphs with large $\Mm$}\label{Sec 3}

\subsection{Graphs having $\mr _+ G=2$}

%\subsection{Some useful constructions}

%Several known constructions can help us to construct matrices with large $\Mm$. Here we list the ones that we will use later in this section. 

Graphs whose minimal rank is small will typically have large $\Mm$. 
We denote by ${\rm mr}_+(G)$ the minimum rank among all positive semidefinite symmetric matrices corresponding to $G$. A characterisation of graphs with $\mr_+(G) \leq 2$ given in \cite{MR2111528} is restated below.

\begin{theorem}[\cite{MR2111528}]\label{thm:BvHL}
Let $G$ be a graph on $n$ vertices. Then ${\rm mr}_+(G)\leq 2$ if and only if $G^c$ has the form  
\begin{equation}\label{mr=2}
(K_{p_1,q_1}\cup K_{p_2,q_2} \cup \ldots \cup K_{p_k,q_k})\vee K_r
\end{equation}
 for appropriate nonnegative integers $k, p_1,q_1,\ldots,p_k,q_k,r$ with $p_i+q_i>0$, $i=1,2,\ldots,k.$
\end{theorem}

In this section we  show that essentially all graphs $G$ with $\mr_+(G)=2$, have the maximal possible $\Mm$. Namely, the main purpose of this section is to show the  following theorem.

\begin{theorem}\label{thm:mr+2}
Let $G$ be a graph on $n$ vertices with $G^c$ of the form: 
 \begin{equation}\label{n-2form}
 (K_{p_0,0}\cup K_{p_1,q_1}\cup K_{p_2,q_2} \cup \ldots \cup K_{p_k,q_k})\vee K_r
 \end{equation}
 where $p_1,q_1,\ldots,p_k,q_k$ are positive integers, $p_0,k,r\geq 0$ and $(p_0,k)\ne (1,1)$. %with $p_i+q_i>0,$ $i=1,2,\ldots,k$.
Then for any two nonnegative integers $n_1$ and $n_2$ that satisfy $4+n_1+n_2=n$ there exists a matrix $A \in S(G)$ %with $G(A)=G$ 
that has precisely two distinct eigenvalues $\lambda$ and $\mu$ with multiplicities $2+n_1$ and $2+n_2$, respectively.
\end{theorem}

\begin{corollary}
 If  $G$ is a graph on $n$ vertices such that $\mr_+(G)=2$ and $G$ is not of the form
$\left( (K_{1,0} \cup K_{p,q})\vee K_r \right)^c$, then $\Mm(G)= \left\lfloor \frac{n}{2} \right\rfloor$.
\end{corollary}

Note that in the case $k=r=0$, we  have $G=(K_{p_0,0})^c=K_{p_0}$ and in the case $p_0=r=0$, $k=1$,  we have
$G=(K_{p,q})^c=K_{p}\cup K_q$.  Matrices in $S(K_n)$ will be the building blocks in our construction. It will prove 
useful that  result from \cite{MR3208410} considers not only the eigenvalues of these matrices, but also says 
something about the pattern of some of the associated eigenvectors. 

\begin{theorem}[\cite{MR3208410}]\label{thm:Knall}
For any given list of real numbers $\sigma=(\lambda_1,\lambda_2,\ldots,\lambda_n)$, $\lambda_1 \neq \lambda_2$, there exists $A_n \in S(K_n)$ with the spectrum $\sigma$.

Furthermore, given any zero-nonzero pattern of a vector in $\R^n$ that contains at least two nonzero elements, $A_n$ can be chosen in such a way that there exist an eigenvector corresponding to $\lambda_1$ with the given pattern.
\end{theorem}

\begin{corollary}
For any integer $n\geq 2$ we have $\Mm(K_n)= \left\lfloor \frac{n}{2} \right\rfloor$.
\end{corollary}

%To start our induction argument we will need to consider $(K_{p,q})^c=K_p \cup K_q$. 
While Theorem \ref{thm:Knall} solves the inverse eigenvalue problem for complete graphs, we note that in order for a matrix $A \in S(K_p \cup K_q)$ to have an eigenvector corresponding to an eigenvalue $\lambda$ that has no zero entries, this eigenvalue will need to be a repeated eigenvalue in $A$. 

%This gives us the following simple corollary. \po{rabimo eigenvectorje??}
%
%\begin{corollary}
%Let $\sigma=(\lambda,\lambda,\mu_1,\ldots,\mu_{n-2})$ be a list of real numbers, where at least two of the $\mu_i$s are not equal to $\lambda.$ Then there exists $A \in S(K_p \cup K_q)$ with the spectrum $\sigma$ and with two orthogonal eigenvectors corresponding to $\lambda$ that do not contain any zero elements. 
%\end{corollary}

We will use inductive arguments to gradually construct matrices with the desired pattern and eigenvalues. 
First we present three lemmas to resolve the cases that will be used as the basis of our construction.

\begin{lemma}\label{ex:K20UK11}
%Modifying Example \ref{ex:K20UKpq}  we can obtain:
There exists a matrix 
in $S((K_{2,0} \cup K_{1,1})^c)$ with two distinct eigenvalues both with multiplicity $2$, and 
with an orthogonal basis of eigenvectors that do not contain any zero elements. 
\end{lemma}

\begin{proof}
The matrix 
{\small $$A=\left(
\begin{array}{cccc}
 \frac{2}{2+3 \sqrt{3}} & \frac{1}{23} \left(9-2 \sqrt{3}\right) & \frac{1}{23} \sqrt{21+26 \sqrt{3}} & \frac{1}{23}
   \sqrt[4]{3} \left(-11+5 \sqrt{3}\right) \\
 \frac{1}{23} \left(9-2 \sqrt{3}\right) & \frac{2}{2+3 \sqrt{3}} & \frac{1}{23} \sqrt{21+26 \sqrt{3}} & \frac{1}{23}
   \sqrt{-330+196 \sqrt{3}} \\
 \frac{1}{23} \sqrt{21+26 \sqrt{3}} & \frac{1}{23} \sqrt{21+26 \sqrt{3}} & -\frac{2}{23} \left(-9+2 \sqrt{3}\right)
   & 0 \\
 \frac{1}{23} \sqrt[4]{3} \left(-11+5 \sqrt{3}\right) & \frac{1}{23} \sqrt{-330+196 \sqrt{3}} & 0 & -\frac{4}{23}
   \left(-9+2 \sqrt{3}\right) \\
\end{array}
\right)$$}
has the desired properties. 
\end{proof}

\begin{lemma}\label{11}
 Let $G$ be a graph on $n \geq 4$ vertices, such that $$G^c=K_{1,1}\cup\ldots\cup K_{1,1}.$$
 Then for all nonnegative integers $n_0,n_1$ satisfying $4+n_0+n_1=n$ there exists a matrix $A \in S(G)$ that has two distinct 
 eigenvalues with multiplicities $2+n_0$ and $2+n_1$.
% Then for all integers $n_0,n_1 \geq 2$ such that   $n_0+n_1=n$ there 
%%  \begin{enumerate}
%  %  \item 
%  exists a matrix $A \in S(G)$   that has eigenvalues $0$ with multiplicity $n_0$ and $1$ with 
%   multiplicity $n_1$. 
%  % \item exists a matrix $A \in S(G)$   that has eigenvalues $\lambda \geq 1$ with multiplicity 2, $0$ with multiplicity $n_0$ and $1$ with multiplicity
  % $n_1$, and two orthogonal eigenvectors corresponding to $\lambda$ that do not contain any zero elements. \he{Ali se vedno rabiva to tocko?}
 %\end{enumerate}
\end{lemma}

\begin{proof}
We are looking for an $n \times n$ matrix $$B \in S\left( (K_{1, 1} \cup \ldots \cup K_{1, 1})^c\right)$$ with eigenvalues $0$ and $1$ with multiplicities $m_0 \geq 2$ and $m_1 \geq 2$, respectively. We will prove this claim by induction.

 In the first step we distinguish between the cases when $m_0$ and $m_1$ are both even and when $m_0$ and $m_1$ are both odd. In the first case let us write $m_0=2 k_0$ and $m_1=2 k_1$ and in the second case let us write $m_0=2 k_0+1$ and $m_1=2 k_1+1.$ Let $k=k_0+k_1.$
In the even case let  
$$A_1=\npmatrix{1+(k-1)t & 0 \\ 0 & 1+(k-1)t} \in S( (K_{1,1})^c)$$ 
and in the odd case let 
$$A_1=\left(
\begin{array}{cccc}
 \frac{2 a}{3}+1 & 0 & -\frac{a}{3} & \frac{1}{3} \sqrt{a^2+\frac{3 a}{2}} \\
 0 & \frac{2 a}{3}+1 & -\frac{a}{3} & -\frac{1}{3} \sqrt{a^2+\frac{3 a}{2}} \\
 -\frac{a}{3} & -\frac{a}{3} & \frac{a+3}{3} & 0 \\
 \frac{1}{3} \sqrt{a^2+\frac{3 a}{2}} & -\frac{1}{3} \sqrt{a^2+\frac{3 a}{2}} & 0 & \frac{a}{3}
   \\
\end{array}
\right) \in S\left((K_{1,1}\cup K_{1,1})^c\right),$$
where  $a=(k-1) t$ and the matrix $A_1$ has 
the spectrum $(1,0,1+(k-1)t,1+(k-1)t).$ It is easy to check that in both cases we can find two orthogonal eigenvectors with no zero entries corresponding to $1+(k-1)t.$

Starting with $A_1$ we will apply Corollary \ref{cor:one_step} in two ways.
For $j=2,\ldots, k_1$ we take $$A_j=\npmatrix{1-t & 0 \\ 0 & 1-t} \in  S( (K_{1,1})^c) \subseteq M_2(\R)$$ and define $C_1=A_1$.  By Corollary \ref{cor:one_step}, we recursively construct matrices $$C_j=\npmatrix{C_{j-1} & S_j \\ S_j^T & A_j} \in S\left( (K_{1, 1} \cup \ldots \cup K_{1, 1})^c\right)\subseteq M_{2j}(\R),$$
where in the even case $C_j$ has eigenvalues  $1+(k-j)t$ with multiplicity $2$ and $1$ with multiplicity $2(j-1)$, and in the odd case $C_j$ has eigenvalues  $1+(k-j)t$ with multiplicity $2$, $1$ with multiplicity $2(j-1)+1$, and $0$ with multiplicity $1$. 
In addition, the construction guarantees that $C_{j}$ has two eigenvectors with no zero elements corresponding to $1+(k-j)t.$
 
 For $j=1,\ldots, k_0$ we take $$A_{k_1+j}=\npmatrix{-t & 0 \\ 0 & -t} \in   S( (K_{1,1})^c) \subseteq M_2(\R).$$ Using  Corollary \ref{cor:one_step} we
recursively define $$C_{k_1+j}=\npmatrix{C_{k_1+j-1} & S_{k_1+j} \\ S_{k_1+j}^T & A_{k_1+j}} \in S\left( (K_{1, 1} \cup \ldots \cup K_{1,1})^c\right)\subseteq M_{2(k_1+j)}(\R)$$ 
such that in the even case $C_{k_1+j}$ has eigenvalues
$1+(k_0-j)t$ with multiplicity $2$, $1$ with multiplicity $2(k_1-1)$ and  $0$ with multiplicity $2j$, and 
in the odd case $C_{k_1+j}$ has eigenvalues
$1+(k_0-j)t$ with multiplicity $2$, $1$ with multiplicity $2(k_1-1)+1$ and  $0$ with multiplicity $2j+1$.

 In this way we  obtain a matrix $C_{k}=C_{k_1+k_0} \in S(K_{1,1} \cup \ldots \cup K_{1, 1})^c$ with eigenvalues $0$ repeated $2k_0$ times and $1$ repeated  $2 k_1$ times in the even case, and  with eigenvalues $0$ repeated $2k_0+1$ times and $1$ repeated  $2 k_1+1$ times in the odd case. 
%The second item in the lemma is proved in the same way, except that we start with initial matrices $A_1$ with the eigenvalues $(\lambda+(1-k)t, \lambda+(1-k)t)$ in the even case and with the eigenvalues  $(\lambda+(1-k)t, \lambda+(1-k)t,1,0)$ in the odd case.  
\end{proof}

\begin{remark}\label{rem:Lema extension}
Note that we can, using the same construction as in the proof above, construct a matrix $A\in S\left((K_{1,1}\cup\ldots\cup K_{1,1})^c\right)$ with eigenvalues $1+t$, $t>0$, with multiplicity 2,
$0$ with multiplicity $n_0$ and $1$ with multiplicity $n_1$ and two orthogonal eigenvectors corresponding to $1+t$ that do not contain any zero elements. The only difference that we need to make in the proof is for the initial matrix $A_1$ to have eigenvalues $(1+k t,1+k t)$ in the even case and eigenvalues  $(1+k t, 1+k t,1,0)$ in the odd case. 
\end{remark}

\begin{lemma}\label{1011}
 Let $G$ be a graph on $n\geq 5$ vertices such that  
 $$G^c=K_{1,0} \cup K_{1,1}\cup\ldots\cup K_{1,1}.$$
 Then for all nonnegative integers $n_0,n_1$ such that   $4+n_0+n_1=n$ there exists a matrix $A \in S(G)$ that has two distinct 
 eigenvalues with multiplicities $2+n_0$ and $2+n_1$.

\end{lemma}

\begin{proof}
Let $n=2s+1$. The matrix

\begin{equation}\label{s=2}
\small{\left(
\begin{array}{ccccc}
 \frac{1}{7} \left(3+\sqrt{3}\right) & 0 & \frac{\sqrt{3}}{7} & \frac{1}{7} \left(1+\sqrt{3}\right) &
   \frac{\sqrt{6}}{7 \left(3+\sqrt{3}\right)} \\
 0 & \frac{1}{7} \left(3-\sqrt{3}\right) & \frac{\sqrt{3}}{7} & \frac{1}{7} \left(1-\sqrt{3}\right) &
   -\frac{\sqrt{6} \left(2+\sqrt{3}\right)}{7 \left(3+\sqrt{3}\right)} \\
 \frac{\sqrt{3}}{7} & \frac{\sqrt{3}}{7} & \frac{3}{7} & 0 & -\frac{\sqrt{6}}{7} \\
 \frac{1}{7} \left(1+\sqrt{3}\right) & \frac{1}{7} \left(1-\sqrt{3}\right) & 0 & \frac{2}{7} & \frac{\sqrt{2}}{7}
   \\
 \frac{\sqrt{6}}{7 \left(3+\sqrt{3}\right)} & -\frac{\sqrt{6} \left(2+\sqrt{3}\right)}{7 \left(3+\sqrt{3}\right)}
   & -\frac{\sqrt{6}}{7} & \frac{\sqrt{2}}{7} & \frac{3}{7} \\
\end{array}
\right)}
%\small{\left(
%\begin{array}{ccccc}
% \sqrt{3} & 0 & \sqrt{3} & 1+\sqrt{3} & \frac{\sqrt{6}}{3+\sqrt{3}} \\
% 0 & -\sqrt{3} & \sqrt{3} & 1-\sqrt{3} & -\frac{\sqrt{6}
%   \left(2+\sqrt{3}\right)}{3+\sqrt{3}} \\
% \sqrt{3} & \sqrt{3} & 0 & 0 & -\sqrt{6} \\
% 1+\sqrt{3} & 1-\sqrt{3} & 0 & -1 & \sqrt{2} \\
% \frac{\sqrt{6}}{3+\sqrt{3}} & -\frac{\sqrt{6}
%   \left(2+\sqrt{3}\right)}{3+\sqrt{3}} & -\sqrt{6} & \sqrt{2} & 0 \\
%\end{array}
%\right)} \in S\left((K_{1,1}\cup K_{1,1} \cup K_{1,0})^c\right)
\end{equation}
is contained in $ S\left((K_{1,1}\cup K_{1,1} \cup K_{1,0})^c\right)$ and has eigenvalues $1$ and $0$ with respective multiplicities $2$ and $3$, and therefore proves our theorem in the case $s=2$.

For $s \geq 3$ consider a graph $$G=(\underbrace{K_{1,1}\cup\ldots\cup K_{1,1}}_{s-2} \cup \left(K_{1,1}\cup K_{1,1}\cup K_{1,0}\right))^c,$$  and take arbitrary integers $\hat n_0, \hat n_1\geq 0$ such that $\hat n_0+\hat n_1+7=n$.
Let $B_1\in S\left((K_{1,1}\cup K_{1,1} \cup K_{1,0})^c\right)$ be a matrix with eigenvalues $-t$, $-t$, $1$, $1$, 
$1$, where $0<t <1$, and two orthogonal eigenvectors corresponding to $-t$ without any zero elements. By example \eqref{s=2} and  \cite[Lemma 2.3]{MR3118943} such a matrix exists.
Remark \ref{rem:Lema extension} gives us a matrix $B_2\in S\left((K_{1,1}\cup\ldots\cup K_{1,1})^c\right)$ on $n-5$ vertices with eigenvalues $1+t$ with multiplicity 2,
$0$ with multiplicity $\hat n_0$ and $1$ with multiplicity $\hat n_1$ and two orthogonal eigenvectors corresponding to $1+t$ that do not contain any zero elements. By Corollary  \ref{cor:one_step} there exists a matrix $C \in S(G)$ with eigenvalues $1$ with multiplicity $5+\hat n_1$ and $0$
with multiplicity $2+\hat n_0$.

We have now covered all the cases, except multiplicities $3, 4$ in the case $s=3$. This case is solved by
 the matrix
 $$
 \small{\left(
\begin{array}{ccccccc}
 \frac{4}{7} & \frac{\sqrt{5}}{7} & \frac{3}{7 \sqrt{5}} & \frac{2}{7 \sqrt{7} \left(1+\sqrt{3}\right)} & -\frac{2
   \left(2+\sqrt{3}\right)}{7 \sqrt{7} \left(1+\sqrt{3}\right)} & -\frac{2 \sqrt{\frac{3}{7}}}{7} & \frac{2}{7 \sqrt{7}}
   \\
 \frac{\sqrt{5}}{7} & \frac{5}{7} & 0 & -\frac{\sqrt{\frac{5}{7}}}{7 \left(1+\sqrt{3}\right)} & \frac{\sqrt{\frac{5}{7}}
   \left(2+\sqrt{3}\right)}{7 \left(1+\sqrt{3}\right)} & \frac{\sqrt{\frac{15}{7}}}{7} & -\frac{\sqrt{\frac{5}{7}}}{7} \\
 \frac{\sqrt{5}}{7} & 0 & \frac{1}{7} & \frac{\sqrt{\frac{5}{7}}}{7+7 \sqrt{3}} & -\frac{\sqrt{\frac{5}{7}}
   \left(2+\sqrt{3}\right)}{7 \left(1+\sqrt{3}\right)} & -\frac{\sqrt{\frac{15}{7}}}{7} & \frac{\sqrt{\frac{5}{7}}}{7} \\
 \frac{2}{7 \sqrt{7} \left(1+\sqrt{3}\right)} & -\frac{\sqrt{\frac{5}{7}}}{7 \left(1+\sqrt{3}\right)} &
   \frac{\sqrt{\frac{5}{7}}}{7+7 \sqrt{3}} & \frac{1}{7} \left(3+\sqrt{3}\right) & 0 & \frac{\sqrt{3}}{7} & \frac{1}{7}
   \left(1+\sqrt{3}\right) \\
 -\frac{2 \left(2+\sqrt{3}\right)}{7 \sqrt{7} \left(1+\sqrt{3}\right)} & \frac{\sqrt{\frac{5}{7}}
   \left(2+\sqrt{3}\right)}{7 \left(1+\sqrt{3}\right)} & -\frac{\sqrt{\frac{5}{7}} \left(2+\sqrt{3}\right)}{7
   \left(1+\sqrt{3}\right)} & 0 & \frac{1}{7} \left(3-\sqrt{3}\right) & \frac{\sqrt{3}}{7} & \frac{1}{7}
   \left(1-\sqrt{3}\right) \\
 -\frac{2 \sqrt{\frac{3}{7}}}{7} & \frac{\sqrt{\frac{15}{7}}}{7} & -\frac{\sqrt{\frac{15}{7}}}{7} & \frac{\sqrt{3}}{7} &
   \frac{\sqrt{3}}{7} & \frac{3}{7} & 0 \\
 \frac{2}{7 \sqrt{7}} & -\frac{\sqrt{\frac{5}{7}}}{7} & \frac{\sqrt{\frac{5}{7}}}{7} & \frac{1}{7} \left(1+\sqrt{3}\right)
   & \frac{1}{7} \left(1-\sqrt{3}\right) & 0 & \frac{2}{7} \\
\end{array}
\right)}
 %  \small{\left(
%\begin{array}{ccccccc}
% 1 & \sqrt{5} & \frac{3}{\sqrt{5}} & \frac{2}{\sqrt{7} \left(1+\sqrt{3}\right)} & -\frac{2 \left(2+\sqrt{3}\right)}{\sqrt{7} \left(1+\sqrt{3}\right)}
%   & -2 \sqrt{\frac{3}{7}} & \frac{2}{\sqrt{7}} \\
% \sqrt{5} & 2 & 0 & -\frac{\sqrt{\frac{5}{7}}}{1+\sqrt{3}} & \frac{\sqrt{\frac{5}{7}} \left(2+\sqrt{3}\right)}{1+\sqrt{3}} & \sqrt{\frac{15}{7}} &
%   -\sqrt{\frac{5}{7}} \\
% \sqrt{5} & 0 & -2 & \frac{\sqrt{\frac{5}{7}}}{1+\sqrt{3}} & -\frac{\sqrt{\frac{5}{7}} \left(2+\sqrt{3}\right)}{1+\sqrt{3}} & -\sqrt{\frac{15}{7}} &
%   \sqrt{\frac{5}{7}} \\
% \frac{2}{\sqrt{7} \left(1+\sqrt{3}\right)} & -\frac{\sqrt{\frac{5}{7}}}{1+\sqrt{3}} & \frac{\sqrt{\frac{5}{7}}}{1+\sqrt{3}} & \sqrt{3} & 0 & \sqrt{3}
%   & 1+\sqrt{3} \\
% -\frac{2 \left(2+\sqrt{3}\right)}{\sqrt{7} \left(1+\sqrt{3}\right)} & \frac{\sqrt{\frac{5}{7}} \left(2+\sqrt{3}\right)}{1+\sqrt{3}} &
%   -\frac{\sqrt{\frac{5}{7}} \left(2+\sqrt{3}\right)}{1+\sqrt{3}} & 0 & -\sqrt{3} & \sqrt{3} & 1-\sqrt{3} \\
% -2 \sqrt{\frac{3}{7}} & \sqrt{\frac{15}{7}} & -\sqrt{\frac{15}{7}} & \sqrt{3} & \sqrt{3} & 0 & 0 \\
% \frac{2}{\sqrt{7}} & -\sqrt{\frac{5}{7}} & \sqrt{\frac{5}{7}} & 1+\sqrt{3} & 1-\sqrt{3} & 0 & -1 \\
%\end{array}
%\right) }
$$
 in $S\left((K_{1,0} \cup K_{1,1}\cup K_{1,1}\cup K_{1,1})^c\right)$ with eigenvalues $1$ and $0$ with respective multiplicities $3$ and $4$.
 %Note that, we can, using the same construction as in the proof above, construct a matrix $A\in S\left((K_{1,1}\cup\ldots\cup K_{1,1})^c\right)$ on $n-5$ vertices with eigenvalues $1+t$ with multiplicity 2,
%$0$ with multiplicity $n_0$ and $1$ with multiplicity $n_1$ and two orthogonal eigenvectors corresponding to $1+t$ that do not contain any zero elements. Contruction is done in the same 
%way as in the proof of Lemma \ref{11}, except that we start with initial matrices $A_1$ with the eigenvalues $(1+(1-k)t,1+(1-k)t)$ in the even case and with the eigenvalues  $(1+(1-k)t, 1+(1-k)t,1,0)$ in the odd case. 
%
\end{proof}

%
%
%\begin{example}\label{5}
%Matrix
%$$\small{\left(
%\begin{array}{ccccc}
% \sqrt{3} & 0 & \sqrt{3} & 1+\sqrt{3} & \frac{\sqrt{6}}{3+\sqrt{3}} \\
% 0 & -\sqrt{3} & \sqrt{3} & 1-\sqrt{3} & -\frac{\sqrt{6}
%   \left(2+\sqrt{3}\right)}{3+\sqrt{3}} \\
% \sqrt{3} & \sqrt{3} & 0 & 0 & -\sqrt{6} \\
% 1+\sqrt{3} & 1-\sqrt{3} & 0 & -1 & \sqrt{2} \\
% \frac{\sqrt{6}}{3+\sqrt{3}} & -\frac{\sqrt{6}
%   \left(2+\sqrt{3}\right)}{3+\sqrt{3}} & -\sqrt{6} & \sqrt{2} & 0 \\
%\end{array}
%\right)} \in S\left((K_{1,1}\cup K_{1,1} \cup K_{1,0})^c\right)$$
%has eigenvalues $4$ with multiplicity 2 and $-3$ with multiplicity $3$. Moreover, there exist two orthogonal eigenvectors
%corresponding to eigenvalue 4 with no zero entries. 
%\end{example}

%\begin{example}\label{ex:K20UKpq}
%Let 
%$$U=\left(
%\begin{array}{cccccccc}
% \frac{1}{2} & \frac{\sqrt{3}}{2} & x & \cdots & x & y & \cdots & y \\
% \frac{\sqrt{3}}{2} & \frac{1}{2} & x & \cdots & x & -y & \cdots & -y \\
%\end{array}
%\right)$$
%where columns with $x$ appear $p$ times and columns with $y$ appear $q$ times. Let us define: 
%$$x=(\frac{3}{4p^2})^{1/4} \text{ and }y=(\frac{3}{q^2})^{1/4}.$$
%Then 
% $$UU^T=\left(
%\begin{array}{cc}
% 1+\frac{3 \sqrt{3}}{2} & 0 \\
% 0 & 1+\frac{3 \sqrt{3}}{2} \\
%\end{array}
%\right)$$
%and $U^TU \in S((K_{2,0} \cup K_{p,q})^c)$ and has eigenvalues $1+\frac{3 \sqrt{3}}{2}$ with multiplicity $2$ and $0$ with multiplicity $p+q.$
%\end{example}

The following technical lemma that is a straightforward application of Lemma \ref{HS04} will enable us to complete the proof of the main theorem of this section.

\begin{lemma}\label{thm:Helena}
Let $\hat{G}$ be a graph on $n$ vertices with $\hat{G}^c$ of the form
 \begin{equation*}
 K_{{p}_0,0}\cup K_{{p}_1,{q}_1}\cup K_{{p}_2,{q}_2} \cup \ldots \cup K_{{p}_k,{q}_k},
 \end{equation*}
 where $k,{p}_1,{q}_1,\ldots,{p}_k,{q}_k$ are positive integers.
Let $G$ be a graph on $n+r\geq n$ vertices, that we obtain from a graph $\hat G$ by increasing one of the parameters $p_j$, $j=0,1,\ldots,k$, by $r$, i.e. $G^c$ is of the form
 \begin{equation*}
 K_{p_0,0}\cup K_{p_1,q_1}\cup \ldots \cup K_{p_{j-1,}q_{j-1}} \cup K_{p_j+r,q_j} \cup  K_{p_{j+1},q_{j+1}}\cup\ldots \cup 
 K_{p_k,q_k}.
 \end{equation*}

If there exists a matrix in $ A \in S(\hat{G})$ with eigenvalues $0$ and $1$ with multiplicities $2+m_0$ and $2+m_1$, where $4+m_0+m_1=n$, such that $A$ does not have any diagonal elements equal to $0$ or $1$, then for any two nonnegative integers $t$ and $s$ satisfying $t+s=r$ there exists a matrix in $S(G)$ 
with eigenvalues $0$ and $1$ with multiplicities $2+m_0+t$ and $2+m_1+s$.
\end{lemma}

\begin{proof}
Let $A \in S(\hat{G})$ have eigenvalues $0$ and $1$ with multiplicities $2+m_0$ and $2+m_1$. Let $a_{jj}$ be one of the diagonal element of $A$ from a block corresponding to $K_{p_j}$. Let  $B \in S(K_{r+1})$ have eigenvalues
$a_{jj}$, 0 with multiplicity $t$ and 1 with multiplicity $s$, $t+s=r$. 
(Here we need the condition on the diagonal elements of $A$, but only in the case when $s=0$ or $t=0$. If for example $s=0$ and $a_{jj}=0$, then matrix $B$ that we need does not exist.)

Furthermore, we demand that the eigenvector corresponding to $a_{jj}$ has no zero entries. Now we apply Lemma \ref{HS04} to join matrices $A$ and $B$ through the diagonal element $a_{jj}.$ The resulting matrix belongs to $S(G)$ and has the desired eigenvalues. 
\end{proof}

\begin{mainproof}
Note that the join with $K_r$ in the statement of the theorem, adds some unconnected vertices to the graph $G$, therefore it is sufficient to study the graphs without the join. 
If $(p_0,k)=(0,1)$, then $G=(K_{p,q})^c=K_p\cup K_q$ and the statement follows by Theorem \ref{thm:Knall}.

Following the proofs of Lemma \ref{ex:K20UK11}, Lemma \ref{11}  and Lemma \ref{1011} it is easy to see that matrices constructed in those lemmas do not have any of the diagonal elements equal to $1$ or to $0$.

Applying Lemma \ref{thm:Helena} to Lemma \ref{ex:K20UK11} proves the theorem for all graphs of the form $$(K_{p_0,0}\cup K_{p_1,q_1})^c,$$  where $p_0 \geq 2,$ $p_1 \geq 1$ and $q_1 \geq 1.$ Lemma \ref{thm:Helena}  together with Lemma \ref{11} gives the result for all graphs of the form $$G=(K_{p_1,q_1}\cup\ldots\cup K_{p_k,q_k})^c,$$ where $p_i \geq 1$ and $q_i \geq 1$. Finally, Lemma \ref{thm:Helena} applied to Lemma \ref{1011} proves the theorem for graphs of the form:
$$G=(K_{p_0,0} \cup K_{p_1,q_1}\cup\ldots\cup K_{p_k,q_k})^c,$$
where $p_i \geq 1$ and $q_i \geq 1$ and $k \geq 2.$
\end{mainproof}

Note that in the case $(p_0,k)= (1,1)$ we have that 
$q\left((K_{1,0} \cup K_{p,q})^c\right)\geq 3$  by \cite[Theorem 3.2]{MR3118943}. Hence
$$q\left((K_{1,0} \cup K_{p,q})\vee K_r\right)^c \geq 3 $$ and this case is excluded in Theorem \ref{thm:mr+2}. It follows from the proposition below that $$\Mm\left((K_{1,0} \cup K_{p,q})\vee K_r\right)^c =\left\lfloor \frac{n}{3}\right\rfloor.$$

\begin{proposition}
 Let $n=p+q+1,$ and let $n_1, n_2, n_3$ be positive integers such that $n_1+n_2+n_3=n.$ Then 
there exists $A \in S((K_{1,0} \cup K_{p,q})^c)$ with eigenvalues $\lambda_1,\lambda_2, \lambda_3$ with multiplicities $n_1, n_2, n_3$, respectivelly. 
\end{proposition}

\begin{proof}
For $n=3$ any matrix in  $S((K_{1,0} \cup K_{p,q})^c)$ satisfies the conditions of the proposition. For example
$$A=\left(
\begin{array}{ccc}
 0 & 2 & 1 \\
 2 & 1 & 0 \\
 1 & 0 & 2 \\
\end{array}
\right) \in S((K_{1,0} \cup K_{1,1})^c).$$
has eigenvalues $(3, -\sqrt{3}, \sqrt{3})$.

Let $p_i$ and $q_i$, $i=1,2,3,$ be nonnegative integers such that $p_1+p_2+p_3=p-1$ and $q_1+q_2+q_3=q-1.$ Let $B_1 \in S(K_p)$ have eigenvalues $1$ with multiplicity $1$, $3$ with multiplicity $p_1$, $\sqrt{3}$ with multiplicity $p_2$ and $-\sqrt{3}$ with multiplicity $p_3$, $1+p_1+p_2+p_3=p$. Moreover, we demand that the eigenvector of $B_1$ corresponding to $1$ has no zero elements. Similarly, let  $B_2 \in S(K_p)$ have eigenvalues $2$ with multiplicity $1$, $3$ with multiplicity $q_1$, $\sqrt{3}$ with multiplicity $q_2$ and $-\sqrt{3}$ with multiplicity $q_3$, $1+q_1+q_2+q_3=q$, with eigenvector corresponding to $2$ having no zero elements. Matrices $B_1$ and $B_2$ exist by Theorem \ref{thm:Knall}. 

Now we use Lemma \ref{HS04} to first join matrix $A$ with $B_1$ through the diagonal element $1,$ and then to join the resulting matrix with $B_2$ through the diagonal element $2$. In this way be obtain a matrix in  $S((K_{1,0} \cup K_{p,q})^c)$
with eigenvalues $3$ with multiplicity $1+p_1+q_1$, $\sqrt{3}$ with multiplicity $1+p_2+q_2$ and $-\sqrt{3}$ with multiplicity $1+p_3+q_3$. Since we didn't pose any conditions on the parameters $p_i$ and $q_i$ this proves or claim. 
\end{proof}

\subsection{Complete bipartite graphs}

Another family of graphs with small minimal rank is the set of complete bipartite graphs. Note that $\mr(K_{m,n})=2$ 
for any positive integers $m,n$. We will prove that 
$\Mm(K_{m,n})=\lfloor \frac{m+n}{3} \rfloor$ if $m \ne n$.

\begin{theorem}\label{Kmn}
Let $m \leq n.$ Then any list of the form
\begin{equation}\label{KmnSpectrum}
 (\lambda_1,-\lambda_1,\lambda_2,-\lambda_2,\ldots,\lambda_m,-\lambda_m,\underbrace{0,\ldots,0}_{n-m}),
\end{equation}
where  $\lambda_1 > 0$ and $\lambda_i \geq 0$ for $i=2,\ldots,m$, % and at least one of the values $\lambda_1,\ldots,\lambda_m$ is different from zero, 
 is the spectrum of a matrix $A \in S(K_{m,n})$.
\end{theorem}

\begin{proof}
Let \begin{equation}\label{eq:BBT}
A=\npmatrix{0_m & B \\ B^T & 0_n},
\end{equation}
where $B$ is an $m \times n$ matrix. Then, using Schur complement, we can compute the characteristic polynomial of $A$ as follows:
 \begin{align*}
 \det(xI-A)&=\det(xI_n)\det(xI_m-x^{-1}BB^T)\\
 &= x^{n-m}\det(x^2I_m-BB^T). 
 \end{align*}
 So, it follows that if $\mu_1, \ldots, \mu_m$ are the eigenvalues of $BB^T$, then
 the eigenvalues of $A$ are $\sqrt{\mu_i},$ $-\sqrt{\mu_i}$ for $i=1,2,\ldots, m,$ together with $n-m$ instances of the eigenvalue $0$. It remains to show that we can find a matrix $B$ that does not contain any zero elements so that $BB^T$ has eigenvalues $\lambda_i$, $i=1,2, \ldots, m$.
 
Let $A_0$ be a symmetric matrix in $S(K_m)$ with eigenvalues $\lambda_1,\lambda_2,\ldots,\lambda_m$ and let $U$ be an orthogonal matrix that diagonalises $A_0$: $U^TA_0U=D$. Then 
 $A_0=UD^{\frac{1}{2}}(UD^{\frac{1}{2}})^T,$
 and the matrix $B_0=UD^{1/2}$ satisfies the spectral conditions that we need, however it can have some zero entries. 
 
Since we have choosen $A_0 \in S(K_m)$ we know that $B_0$ has no zero rows, so there exists an orthogonal matrix $V$ such that $B_0V$ has no zero elements. (A generic $n \times n$ orthogonal matrix will accomplish this.) We take $B=B_0V$ to finish the proof.
\end{proof}

\begin{corollary}
For any positive integers $m,n$,
 $$\Mm(K_{m,n})=\begin{cases}
   m, & m=n,\\
    \lfloor \frac{m+n}{3} \rfloor, &\text{otherwise.}
 \end{cases}$$
\end{corollary}

\begin{proof}
 Note that by \cite[Corollary 6.5]{MR3118943}, we have
 $$q(K_{m,n})=\begin{cases}
   2, & m=n,\\
    3, & m \ne n.
 \end{cases}$$
From \eqref{eqn:basic} it follows  that  $\Mm(K_{m,m})\leq m$ and $\Mm(K_{m,n})\leq\lfloor\frac{m+n}{3}\rfloor$ for $m \ne n$. In the case $m=n$ this implies by Theorem \ref{Kmn} that $\Mm(K_{m,m})=m$.

 If $m \ne n$, let 
 $m+n=3s+k,$ where $k \in \{0,1,2\}$. If $k\in \{0,1\}$, then choose in Theorem \ref{Kmn} a matrix $A \in K_{m,n}$ with eigenvalues $\lambda$, $-\lambda$ and $0$ with  multiplicities $s$, $s$,
 and $s+k$, respectively. In the case $k=2$, take  $A \in K_{m,n}$ with eigenvalues $\lambda$, $-\lambda$ and $0$ with  multiplicities $s+1$, $s+1$ and $s$, respectively.
 This shows that $\Mm(K_{m,n})=s= \lfloor \frac{m+n}{3} \rfloor$ if $m \ne n$.
%
% $m+n=3\lfloor \frac{m+n}{3}\rfloor+k$. If $k\in \{0,1\}$, then choose in Theorem \ref{Kmn} a matrix $A \in K_{m,n}$ with eigenvalues $\lambda$, $-\lambda$ and $0$ with  multiplicities $\lfloor \frac{m+n}{3}\rfloor$, $\lfloor \frac{m+n}{3}\rfloor$,
% and $\lceil \frac{m+n}{3}\rceil$, respectively. In the case $k=2$, take  $A \in K_{m,n}$ with eigenvalues $\lambda$, $-\lambda$ and $0$ with  multiplicities $\lceil \frac{m+n}{3}\rceil$, $\lceil \frac{m+n}{3}\rceil$ and $\lfloor \frac{m+n}{3}\rfloor$, respectively.
% This shows that $\Mm(K_{m,n})= \lfloor \frac{m+n}{3} \rfloor$ if $m \ne n$.
\end{proof}

%%%%%%%%%%%%
\section{Graphs with small $\Mm$}\label{Sec 4}

In this section we look at some cases when $\Mm$ is small. Since we know that the largest and the smallest eigenvalue of a matrix $A \in S(T)$, where $T$ is a tree, both have multiplicities 1, \cite[Remark 4]{MR1902112}, we have 
$\Mm(T)= 1$ for all trees $T$.

In \cite{MR2549052} a characterisation of graphs with maximal multiplicity of an eigenvalue
equal to 2 is given. It is shown that if in the graph $G$ there exist two
independent induced paths that cover all the vertices of $G$ and such that any
edges between the two paths can be drawn so that they do not cross, then $G$ has maximal multiplicity of an eigenvalue
equal to 2. Such graphs are called \emph{graphs with two parallel paths}. It clear, that if $\mr(G)=n-2$, then $\Mm(G)$ is either equal to $1$ or to $2$. We will show that both cases can occur. 

% This is not if and only if. There exist 6 exceptional graphs having M=2., see \cite[page 736]{MR2549052}. }

%\begin{example}
% Note that $M(C_n) = 2$, which implies that $\Mm(C_n)\leq 2$. Thus, if $n$ is odd, $\Mm(C_{n})=1$. In \cite[Corollary 3.1]{MR3208410} we proved that in the case $n$ is even there exists matrix $A \in S(C_{n})$ having multiplicities of eigenavalues all even and
% so $\Mm(C_n)=2$ if $n$ even.  
%\end{example}

First, we start by Theorem that will imply $\Mm(G)=1$ for certain unicyclic graphs $G$ having $M(G)=2$.

\begin{lemma}\label{prop:arrow}
Let 
\begin{equation}\label{arrow}
A=\npmatrix{d & a^T \\ a & D},
\end{equation}
 where $a=\npmatrix{a_1 & a_2 & \ldots & a_n}^T$ and $D$ is a diagonal matrix with diagonal elements $d_1,d_2,\ldots,d_n$. Let $\lambda$ be an eigenvalue of $A$ with multiplicity $k$. 
If the multiplicity of $\lambda$ in $D$ is either $k-1$ or  $k$, then $d_i=\lambda$ implies $a_i=0.$
 \end{lemma}
 
 \begin{proof}
The characteristic polynomial of $A$ is equal to
  $$p(x)=(x-d)\prod_{i=1}^n(x-d_i)-\sum_{i=1}^n a_i^2 \prod_{j \neq i}(x-d_j).$$
Since $\lambda$ is an eigenvalue of $A$ of multiplicity $k$, it follows that $p(x)=(x-\lambda)^kp_1(x)$, where $p_1(\lambda)\neq 0$. By interlacing, $\lambda$ is an eigenvalue of $D$ with multiplicity $l$ , where  $k-1\leq l \leq k+1$. 
Suppose that $l \leq k$ and without loss of generality we let $d_1=d_2=\ldots=d_{l}=\lambda$ and we assume $d_{l+1}, \ldots,d_n$ are all different from $\lambda$. Then we have:
\begin{align*}
(x-\lambda)^{k-l+1}p_1(x)&=(x-d)(x-\lambda)\prod_{i=l+1}^n(x-d_i)-\\ &-\sum_{i=1}^l a_i^2\prod_{j=l+1}^n(x-d_j)-\sum_{i=l+1}^l a_i^2(x-\lambda)\prod_{
\substack{
 j=l+1\\
 j\neq i}}^n(x-d_j).
\end{align*}
Since $k-l+1 \geq 1$, all but one therm in the above expression is divisible by $(x-\lambda)$. The therm  that is not divisible by $(x-\lambda)$ is:
$$\sum_{i=1}^l a_i^2\prod_{j=l+1}^n(x-d_i)=\left(\sum_{i=1}^l a_i^2\right)\prod_{j=l+1}^n(x-d_i).$$
This implies that $\sum_{i=1}^l a_i^2=0$ and $a_1=\ldots=a_l=0$.
 \end{proof}
 
% \todo{For a related statement look at Theorem 11, in Johnson et al.: LAA 373, 2003.}

 \begin{theorem}\label{thm:star}
Let $G$ be a graph with $A \in S(G)$ of the form
\begin{equation}\label{structure1} 
 A=\npmatrix{d & b_1^T & \ldots &b_p^T & c^T \\
                          b_1 & B_1 & &&&  \\
                          \vdots&&\ddots&&&\\
                          b_p &  & &B_p&  \\
                          c & &&& D} \in \R^{n \times n},
\end{equation}   
where $B_i \in \R^{n_i \times n_i}$, $n_i \geq 2$, are not diagonal matrices, and $D$ is an $m \times m$ diagonal matrix. Moreover, $b_i \in \R^{n_i}$ are not zero vectors and $c \in \R^{m}$ has no zero elements.
Let matrix $A$ be of the form \eqref{structure1} and have $t$ distinct eigenvalues. Then  $$\Mm(A) \leq \frac{2(n-m-p-1)}{t+1}+1.$$ In particular: $$\Mm(G) \leq \frac{n-m-p-1}{2}+1.$$
\end{theorem}
 
 \begin{proof}
 Let $A$ have $t$ distinct eigenvalues $\lambda_i$, $i=1,2,\ldots,t$, with multiplicities $k_1, k_2,\ldots,k_t$. If $\min_i \{k_i\}=1$, then $\Mm(A)=1$ and the statement follows. So, suppose $k_i \geq 2$ for
 $i=1,2,\ldots,t$. By the interlacing theorem each $\lambda_i$ is an eigenvalue of the submatrix  $B\oplus D=B_1\oplus\ldots\oplus B_p \oplus D$ of multiplicity $k_i-1$, $k_i$ or $k_i+1$. Without loss of generality we may assume that $\lambda_i$, $i=1,2,\ldots,s$, have multiplicities $k_i-1$ or $k_i$ as eigenvalues of $B\oplus D$, and eigenvalues $\lambda_i$, $i=s+1,\ldots,t$, have multiplicities $k_i+1$ as eigenvalues of 
 $B\oplus D$. Denote $s'=t-s$. It follows that
  \begin{equation*}
   k_1+k_2+\ldots+k_s-s+k_{s+1}+\ldots +k_t+s' \leq \sum_{i=1}^t {\rm mult}_{B\oplus D} (\lambda_i) \leq n-1,
  \end{equation*}
   and from here $s-s'\geq 1$. From $t=s+s'$ it now follows that $s \geq \frac{t+1}{2}$.
   
Let $U_B=U_1\oplus\ldots\oplus U_p$ be an orthogonal matrix that diagonalises $B=B_1\oplus\ldots \oplus B_p$, i.e. $U_BBU_B^T=D_{1}\oplus \ldots \oplus D_p$, where $D_i\in \R^{n_i \times n_i}$ 
are all diagonal matrices. Then $U=1 \oplus U_B \oplus I_{m}$ puts $A$ in the form of Lemma \ref{prop:arrow}:
$$UAU^T=\npmatrix{d & b_1^T U_1^T& \ldots &b_{p}^TU_p^T & c^T \\
                          U_1b_1 & D_1 & &&  \\
                          \vdots&&\ddots&&\\
                          U_pb_{p} && &D_{p} & \\
                          c & &&& D}.
$$
If for some $j$, $j \in \{1,2,\ldots,p\}$, all the eigenvalues of $B_j$ are among $\lambda_i$, $i=1,2,\ldots,s$, then $U_jb_j=0$ by Lemma \ref{prop:arrow} and hence $b_j=0$, a contradiction. 
So for each $j$, $j \in \{1,2,\ldots,p\}$, at least one of the eigenvalues of $B_j$ has to be different than $\lambda_i$, $i=1,2,\ldots,s$. Since we are assuming that $c^T$ has no zero elements, Lemma \ref{prop:arrow} also tells us that $\lambda_i$, $i=1,\ldots, s$, 
are not eigenvalues of $D$.  Since $k_i \geq 2$, $\lambda_1,\ldots,\lambda_s$ are the eigenvalues of $B$ with multiplicities at least $k_i-1\geq 1$. 
The sum of multiplicities of  $\lambda_1$, \ldots, $\lambda_s$ in $B$ gives us
 $$k_1+k_2+\ldots+k_s-s \leq n_1+\ldots+n_p-p.$$
Since $\Mm(A) \leq k_i$ for $i=1,\ldots,s$, we have
 $$s\, \Mm(A)  \leq n_1+\ldots+n_p-p+s=n-m-p-1+s$$
 and therefore
 $$\Mm(A)   \leq \frac{n-m-p-1}{s}+1  \leq \frac{2(n-m-p-1)}{t+1}+1.$$
 In particular, by \cite[Theorem 3.2]{MR3118943} the minimal number of distinct eigenvalues of $A$ is $q(G) \geq 3$ and thus we have $t\geq 3$ and the result is proved.                     
%
% $s(\Mm(G)-1) +1 \leq n_1$. 
%
%Now we have: 
%  \begin{align*}
%  2 n_1 &\geq s(\Mm(G)-1)+1+s(\Mm(G)-1)+1\\
%  &\geq s(\Mm(G)-1)+1+(s'+1)(\Mm(G)-1)+1\\
%  &=(s+s')(\Mm(G)-1)+\Mm(G)+1 \\
%  &\geq 4 \Mm(G)-2,  
%  \end{align*}
 \end{proof}

  \begin{corollary}\label{cor:small}
Let $G$ be a graph with $A \in S(G)$ of the form
\begin{equation}\label{structure} 
 A=\npmatrix{d & b^T & c^T \\
                          b & B & 0 \\
                          c & 0 & D} \in \R^{n \times n},
\end{equation}   
where $B\in M_{n_1 \times n_1}(\R)$, $D$ is  an $n_2 \times n_2$ diagonal matrix, $b \in \R^{n_1}$, b$\neq 0$, and $c \in \R^{n_2}$ has no zero elements. Then $$\Mm(G) \leq \frac{n_1+1}{2}.$$ 
\end{corollary}

 \begin{remark}
 The inequality between $n_1$ and $\Mm(G)$ is the best possible, since it is achieved for complete graphs $G$. (In this case $n_2=0$.) It is not difficult to find examples of $A \in S(G)$ for which the inequality is achieved, where $G(B)$ is the complete graph on $n_1$ vertices and $n_2$ is kept general.
 \end{remark}
 
 \begin{remark}
For $n_1=2$ in Corollary \ref{cor:small} we get $\Mm(G)=1$. If $B$ has nonzero off-diagonal entries, then the theorem gives us an example
 of a graph $G$ on $n$ vertices with $\Mm(G)=1$ that is not a tree, and whose maximal multiplicity of an eigenvalue is  $n-3$.
 Example of such graph where $n_1=2$ and $n_2=7$:
 
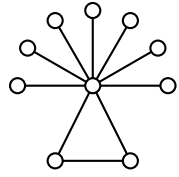
\begin{center}
   \begin{tikzpicture}[style=thick]
		\draw \foreach \x in {0,30,60,90,120,150,180} {
				(0:0) node{} -- (\x:1) node{}
		};		
		\draw[fill=white] \foreach \x in {0,30,60,90,120,150,180}{
		                    (\x:1) circle (1mm)
		};
		\draw (0:0)  -- (0.5,-1) -- (-0.5,-1) -- (0,0);
		\draw[fill=white] \foreach \x in {(0,0),(0.5,-1),(-0.5,-1)}{
			\x circle (1mm)
		};		
   \end{tikzpicture} 
   \end{center}
\end{remark}

 \begin{example}
 Let  $a^T=(0,0,\sqrt{\frac{1}{m+1}})$, $c^T=(\sqrt{\frac{1}{m+1}},\sqrt{\frac{1}{m+1}},\ldots,\sqrt{\frac{1}{m+1}})\in \R^m$ and 
 $$A_1=\left(
\begin{array}{ccc}
 1 & 0 & 0 \\
 0 & -1 & 0 \\
 0 & 0 & 0 \\
\end{array}
\right).$$
The matrix:
 $$A=\npmatrix{0 & a^T & c^T \\
                      a & A_1 & 0 \\
                      c & 0 & 0}\in \R^{(m+4)\times (m+4)}$$
has the spectrum $(1,1,-1,-1,0,0,\ldots,0),$ where the multiplicity of $0$ is $m.$ The matrix
$$U=\left(
\begin{array}{ccc}
 \frac{1}{\sqrt{3}} & \frac{1}{\sqrt{3}} & \frac{1}{\sqrt{3}} \\
 -\sqrt{\frac{2}{3}} & \frac{1}{\sqrt{6}} & \frac{1}{\sqrt{6}} \\
 0 & -\frac{1}{\sqrt{2}} & \frac{1}{\sqrt{2}} \\
\end{array}
\right).$$
is an orthogonal matrix,
 $$B=UA_1U^T=\left(
\begin{array}{ccc}
 0 & -\frac{1}{3 \sqrt{2}}-\frac{\sqrt{2}}{3} & \frac{1}{\sqrt{6}} \\
 -\frac{1}{3 \sqrt{2}}-\frac{\sqrt{2}}{3} & \frac{1}{2} & \frac{1}{2 \sqrt{3}} \\
 \frac{1}{\sqrt{6}} & \frac{1}{2 \sqrt{3}} & -\frac{1}{2} \\
\end{array}
\right) \in S(K_3)$$  and  $$Ua=\left(
\begin{array}{c}
 \frac{1}{\sqrt{3} \sqrt{m+1}} \\
 \frac{1}{\sqrt{6} \sqrt{m+1}} \\
 \frac{1}{\sqrt{2} \sqrt{m+1}} \\
\end{array}
\right).$$
The matrix $$\npmatrix{0 & a^TU^T & c^T \\
                                       Ua & B & 0 \\
                                       c & 0 & 0}$$
is similar to $A$, hence it has the same spectrum. 

This example shows that the inequality in Corollary \ref{cor:small} can be achieved, and that the multiplicity of one of the eigenvalues can be arbitrarily large. 
 \end{example}

Next we introduce a family of graphs $G$ with $M(G)=2$ and $\Mm(G)=2$.

\begin{proposition}\label{squares}
Let 
 $$M_{2n}=\npmatrix{T_n & D_n \\ D_n & T_n},$$
 where $T_n$ is an $n \times n$ tridiagonal matrix with all its nonzero entries equal to $1$, and $D_n$ a diagonal matrix with 
 diagonal entries $d_1,\ldots,d_n$ satisfying the condition $$d_j=-d_{n-j+1}$$ for $j=1,2,\ldots, \lceil \frac{n}{2} \rceil$. 
 Then the characteristic polynomial of $M_{2n}$ is of the form $p(x)^2$ for some polynomial $p(x)$ of degree $n$. 
\end{proposition}

\begin{proof}
First we observe that 
$$\npmatrix{I_n & 0_n \\ I_n & I_n}\npmatrix{T_n & D_n \\ D_n & T_n}\npmatrix{I_n & 0_n \\ -I_n & I_n}=
\npmatrix{T_n-D_n & D_n \\ 0_n & T_n+D_n}.$$
%and thus
%$$\det(x I_{2n}-M_{2n})=\det(xI_n-(T_n-D_n))\det(xI_n-(T_n+D_n)).$$
Let $P_n$ be a permutation matrix with $p_{j,\, n-j+1}=1$ for $j=1,2,\ldots n$. Notice that 
$P_nT_nP_n=T_n$, $P_nD_nP_n=-D_n$, and
 $$P_n(T_n+D_n)P_n=T_n-D_n.$$
 %so $\det(xI_n-(T_n+D_n))=\det(xI_n-(T_n-D_n))$
This shows that $T_n-D_n$ and $T_n+D_n$ are similar, hence have the same spectrum.
\end{proof}

\begin{corollary}\label{parallel}
Let $G$ be a graph with $M_{2n}\in S(G)$, where $M_{2n}$ is defined in Proposition \ref{squares}. Then $\Mm(G)=2$.
\end{corollary}

\begin{proof}
The result follows from the fact that graphs $G$ with $M_{2n} \in S(G)$ are graphs with two parallel paths and therefore
$M(G)=2$. 
By Proposition \ref{squares} it follows that  $\Mm(G)=2$.
\end{proof}

\begin{example}
The four connected graphs on 8 vertices having $\Mm(G)=2$ that are covered by Corollary \ref{parallel} are:
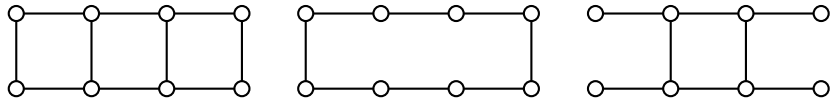
\begin{center}
\begin{tikzpicture}[style=thick,scale=1]
		\draw \foreach \x in {0,1,2} {
				(\x,0) node{} -- (\x,1) node{}
				(\x,0) node{} -- (\x+1,0) node{}				
				(\x,1) node{} -- (\x+1,1) node{}				
		};		
		\draw (3,0) -- (3,1);
		\draw[fill=white] \foreach \x in {0,1,2,3} {
		                    (\x,0) circle (1mm)
       		                    (\x,1) circle (1mm)
		};		
\end{tikzpicture}  
\quad
\begin{tikzpicture}[style=thick,scale=1]
		\draw \foreach \x in {0,1,2} {
				(\x,0) node{} -- (\x+1,0) node{}				
				(\x,1) node{} -- (\x+1,1) node{}				
		};		
		\draw (3,0) -- (3,1);
		\draw (0,0) -- (0,1);
		\draw[fill=white] \foreach \x in {0,1,2,3} {
		                    (\x,0) circle (1mm)
       		                    (\x,1) circle (1mm)
		};		
\end{tikzpicture}  
\quad
\begin{tikzpicture}[style=thick,scale=1]
		\draw \foreach \x in {0,1,2} {
				(\x,0) node{} -- (\x+1,0) node{}				
				(\x,1) node{} -- (\x+1,1) node{}				
		};		
		\draw (1,0) -- (1,1);
		\draw (2,0) -- (2,1);
		\draw[fill=white] \foreach \x in {0,1,2,3} {
		                    (\x,0) circle (1mm)
       		                    (\x,1) circle (1mm)
		};		
\end{tikzpicture}  
\end{center}
\end{example}

In this work we introduced a parameter $\Mm$ to be studied in connection with the inverse eigenvalue problem for graphs. We listed some basic properties of $\Mm$  and derived $\Mm(G)$ for some families of graphs $G$. We believe that looking at $\Mm$ is a good way to expose a deeper insight into the eigenvalue structure that is allowed under pattern constraints imposed by a given graph $G$.  




%\begin{thebibliography}{00}

\end{document}